 \documentclass[12pt, reqno]{amsart}
 \usepackage{amsmath, amsthm, amscd, amsfonts, amssymb, graphicx, color}
 \usepackage[bookmarksnumbered, colorlinks, plainpages]{hyperref}

 \newtheorem{theorem}{Theorem}[section]
 \newtheorem{lemma}[theorem]{Lemma}
 
 \newtheorem{corollary}[theorem]{Corollary}
 \theoremstyle{definition}
 \newtheorem{definition}[theorem]{Definition}

 \theoremstyle{remark}
 \newtheorem{remark}[theorem]{Remark}
 \numberwithin{equation}{section}

\begin{document}
    \setcounter{page}{1}
\title[Automorphisms of
linear functional graphs over vector spaces]{Automorphisms of linear
functional graphs over vector spaces}

\author[ A. Majidinya ]{ Ali Majidinya }

\address{Department of
Mathematical Sciences, Salman Farsi University of Kazerun, Kazerun,
Iran, P.O. Box 73175-457.}
\email{\textcolor[rgb]{0.00,0.00,0.84}{ali.majidinya@gmail.com and
ali.majidinya@kazerunsfu.ac.ir}}



\begin{abstract}
Let $\mathbb{F}_q$ be a finite field with $q$ elements, $n\geq2$ a
positive integer, $\mathbb{V}_0$ a $n$-dimensional vector space over
$\mathbb{F}_q$ and $\mathbb{T}_0$ the set of all linear functionals
from $\mathbb{V}_0$  to $\mathbb{F}_q$. Let
$\mathbb{V}=\mathbb{V}_0\setminus\{0\}$ and
$\mathbb{T}=\mathbb{T}_0\setminus\{0\}$. The \emph{linear functional
graph} of $\mathbb{V}_0$ dented by $\digamma(\mathbb{V})$, is an
undirected bipartite graph, whose vertex set $V$ is partitioned into
two sets as $V=\mathbb{V}\cup \mathbb{T}$  and two vertices $v\in
\mathbb{V}$ and $f\in \mathbb{T}$ are adjacent if and only if $f$
sends $v$ to the zero element of $\mathbb{F}_q$ (i.e. $f(v)=0$). In
this paper, the structure of all automorphisms of this graph is
characterized and formolized. Also the cardinal number of
automorphisms group
 for this graph is determined.
\end{abstract}
\maketitle
\section*{\bf Introduction}
There are many investigations and studies on various graphs
associated to algebraic structures, for example, graphs associated
to the vector spaces, modules, rings and groups, for instance see
\cite{ak2}, \cite{Das}, \cite{Matczuk1} and \cite{wang2}. There are
results about the graphs associated to symplectic spaces, orthogonal
spaces or unitary spaces over finite fields, we refer the reader to
\cite{Gu1} and \cite{Gu2}. One of the most important object for the
graphs related to algebraic structures, is the automorphism group of
these graphs. There are many studies about the automorphism groups
of
 the graphs see \cite{Gu3}, \cite{Wliu1},
 \cite{ZWan1}, \cite{LWang1} and \cite{LWang2}. In particular for
automorphisms of the graphs related to the vector spaces see
\cite{XWang1} and references therein. Wong et al. in \cite{wong1}
have characterized the automorphisms of the zero-divisor graph,
whose vertex set consists of all rank one upper triangular matrices
over a finite field.\par
 In \cite{XWang1} Wang et al. have studied  the
transformation graphs of vector spaces. They investigated the
problem of whether or not a linear transformation sends a vector to
the zero vector. The problem was interpreted by language of graph
theory more explicitly. They defined the transformation graph of a
vector space over a finite field and studied the structure
parameters of this graph, like diameters, domination numbers  and
automorphisms. In {\cite[page13, part(c) ]{XWang1}} the authors have
asked a question about the structure of the automorphisms of a graph
related to linear functionals of the vector spaces and have
interested in the problem of whether or not a linear functional of a
vector space sends a vector to the zero.\par

Following \cite{XWang1}, in \cite{wang1} Wang defined the dual graph
of vector space $\mathbb{V}_0$ over a finite field $\mathbb{F}_q$
denoted by $DG(\mathbb{V})$, with the bipartite two coloring vertex
set $V=X\cup X^*$,
 where $X$ is the set of one-dimensional subspaces of $\mathbb{V}_0$ and $X^*$
 is the set of one-dimensional subspaces of dual space of
 $\mathbb{V}_0$ and two vertices $S\in X$ and $T\in X^*$ are adjacent if and only if
  $f(s)=0$ for all $f\in T$ and all $s\in S$. In
 \cite{XWang1} the author have determined the domination number, independence
number, diameter and girth of $DG(\mathbb{V})$, respectively, also
 such a graph is proved to be distance transitive.\par
Note that the vertex set of dual graph $DG(\mathbb{V})$ in
\cite{wang1} consists of one-dimensional subspaces. But the vertex
set in \cite{XWang1} consists of nonzero vectors and nonzero
transformations.
\par
 Following {\cite[page13, part(c) ]{XWang1}}
  the aim of this paper is to
introduce a graph related to the linear functionals on a vector
space $\mathbb{V}_0$ over a finite field $\mathbb{F}_q$, said to be
the \emph{linear functional graph of the vector space
$\mathbb{V}_0$}, which is
 denoted by $\digamma(\mathbb{V}) $, where $\mathbb{V}=\mathbb{V}_0\setminus
 \{0\}$. It is
interesting to think about the problem whether or not a linear
functional sends a vector to the zero, in the language of graph
theory to investigate  the {\cite[page13, part(c) ]{XWang1}} and
{\cite[page12, Remark ]{wang1}}. The main role of cardinal number
for a finite group is well-known, so determining the cardinal number
for the automorphism group for these kinds of finite graphs is
another motivation to define the linear functional graphs. The first
introductional section contains definitions, elementary observations
( like domination number, connectivity and regularity of the graphs)
and some needed later on. In section 2 the concentrating is on
characterizing the structure of automorphisms for these graphs and
this is done in theorems \ref{thn=2} and \ref{mainth2} for
2-dimensional and $n$-dimensional vector spaces ($n\geq3$),
respectively. The cardinal number of the automorphism group of
$\digamma(\mathbb{V})$ is determined in section 3 according to the
theorems \ref{card2} and \ref{permutSNS}. Finally, the section 4
concentrates on formolizing all of the elements of automorphism
group of the graph $\digamma(\mathbb{V})$(see theorems
\ref{n=3theorem} and \ref{formn=2}).
\section{\bf Preliminaries}
Let $\digamma=(V,E)$ be a simple graph with the vertex set $V$ and
edge set $E$. Then
 $\deg_{\digamma}(v)$ stands for the  \it degree \rm of $v\in V$, i.e.
the cardinality of the set of all vertices  which are adjacent to
$v$. It is written $u\sim v$ if $u$ and $v$ are  adjacent vertices
in the graph. A graph $\digamma=(V, E)$ is said to be a\emph{
bipartite graph} with vertex bipartition $V=X\cup Y$ if every edge
of the graph has one end in $X$ and another end in $Y$. A subset $D$
of the vertices of the graph $\digamma$ is called a \emph{dominating
set}, if every vertex in $V\setminus D$ is adjacent to at least one
vertex of $D$. The minimum size of such a subset is called the
\emph{domination number} of $\digamma$. For a bipartite graph
$\digamma$ with vertex bipartition $V=X\cup Y$, a subset $Z$ of $X$
is said to be a dominating set of $Y$ if any vertex in $Y$ is
adjacent to at least one vertex of $Z$, and the domination number of
$Y$ in this bipartite graph is the minimum size of a dominating set
of $Y$.
 For a positive integer $k$, the graph $\digamma=(V, E)$ is said to
  be \it  $k$-regular\rm , if $\deg_{\digamma}(v)=k<\infty$
 for every $v\in V$.
For the positive integers $m,n$ the \emph{complete  bipartite} graph
$\mathcal{K}_{m,n}$ is a bipartite graph, with vertex bipartition
$V=X\cup Y$, such that $|X|=m$ and $|Y|=n$ and for every two
vertices $x\in X $ and $y\in Y$, $x$ and $y$ are adjacent. Two
simple nonempty graphs $\digamma=(V,E)$ and $\digamma'=(V',E')$ are
said to \emph{isomorphic} if there exists a one to one
correspondence $\rho:V\rightarrow V'$, such that for every  $x,y\in
V$, $x$ and $y$ are adjacent in $\digamma$ if and only if $\rho(x)$
and $\rho(y)$ are adjacent in $\digamma'$. Such a $\rho$ is said to
be a \emph{graph isomorphism} from $\digamma$ to $\digamma'$. If
$\rho$ is an isomorphism form the graph $\digamma$ to $\digamma$,
then $\rho$ is said to be an \emph{automorphism} of the graph
$\digamma$. The subgraph of $\digamma=(V,E)$ induced by a nonempty
subset $A\subseteq V$ is denoted by $\langle A\rangle$.\par
Trough this paper assume $\mathbb{F}_q$ is a finite field with $q$
elements, $n\geq 2$ is an integer and $\mathbb{V}_0=\mathbb{F}_q^n$
is the $n$-dimensional column  vector space on $\mathbb{F}_q$. Let
$\mathbb{T}_0$ be the set of all linear functionals from
$\mathbb{V}_0$ to $\mathbb{F}_q$. Let $e_i$, with $1\leq i\leq n$,
be the vector in $\mathbb{F}^n_q$ whose $i$th entry  is 1 and all
other entries are 0. Any vector $v\in \mathbb{V}$ can be written as
$v = \sum^n_{i=1} a_ie_i$ with $a_i\in \mathbb{F}_q$.
 Then for an element
$u=\sum^n_{i=1} u_ie_i\in \mathbb{V}_0$ consider the map
$f_u:\mathbb{V}_0\rightarrow \mathbb{F}_q$ with
$f_u(x)=u_1x_1+u_2x_2+\cdots+ u_nx_n$, for every $x=\sum^n_{i=1}
x_ie_i\in \mathbb{V}_0$. Clearly $f_u$ is a linear functional on
$\mathbb{V}_0$. It is known that every linear functional from
$\mathbb{V}_0$ to $\mathbb{F}_q$ is of the form $f_u$ for a unique
$u\in \mathbb{V}_0$(see\cite{Hoffman}). Indeed, for every $u,x \in
\mathbb{V}_0$, $f_u(x)=u\cdot x$, the dot product( standard inner
product) of the vectors $u$ and $x$ in the vector space
$\mathbb{V}_0$. Moreover, if we consider the elements of
$\mathbb{V}_0$ as $n\times 1$ matrices such a $u=[u_1\,\, u_2\ldots
u_n]^T$, then $f_u(x)=u^Tx$, the product of two matrices $u^T$ and
$x$, where $u^T$ is the transpose of the matrix $u$. Trough this
paper we also use the notations $\mathbb{V}_0$,
$\mathbb{V}=\mathbb{V}_0\setminus \{0\}$, $\mathbb{T}_0$ and
$\mathbb{T}=\mathbb{T}_0\setminus \{0\}$ as above known notations.
In  a  vector  space  of  dimension \emph{n},  a subspace  of
dimension (\emph{n-}1) is called a \emph{hyperspace}. Such spaces
are sometimes called hyperplanes or subspaces of codimension 1. For
more definitions see \cite{Hoffman}.
\begin{remark}
Note that $\mathbb{T}_0$ is a $n$-dimensional vector space over
$\mathbb{F}_q$ with the vector addition $f_u+f_v=f_{u+v}$ and the
scalar multiplication  $rf_u=f_{ru}$ for every $r\in \mathbb{F}_q$
and $u,v\in \mathbb{V}_0$. Moreover,
$|\mathbb{T}_0|=|\mathbb{V}_0|=q^n$ and $\mathbb{T}_0$ and
$\mathbb{V}_0$ are isomorphic vector spaces as
$\mathbb{F}_q$-modules.
\end{remark}

\begin{definition}\label{a1} Let $n\geq 2$ be an integer,
$\mathbb{F}_q$ a field with $q$ elements,
$\mathbb{V}_0=\mathbb{F}^n_q$ as vector space over the field
$\mathbb{F}_q$ and $\mathbb{T}_0$ the set of linear functionals from
$V_0$ into $\mathbb{F}_q$. The \emph{linear functional graph} of
 $\mathbb{V}_0$,
 denoted by $\digamma(\mathbb{V})$, is a bipartite graph, whose vertex set $V$
  is partitioned into
 sets as $V=\mathbb{T}\cup \mathbb{V}$ such that
$\mathbb{V}=\mathbb{V}_0\setminus \{0\}$ is the set of all nonzero
vectors of $\mathbb{V}$ and $\mathbb{T}=\mathbb{T}_0\setminus\{0\} $
is the set of all nonzero linear functionals on $\mathbb{V}_0$,
where two vertices $f_u\in \mathbb{T}$ and $v\in \mathbb{V}$ are
adjacent if and only if $f_u(v)=0$ (i.e. $v \in ker(f_u)$ the null
space of the linear functional $f_u$).
\end{definition}
\begin{lemma}\label{regular}
The degree of every vertex of $\digamma(\mathbb{V})$ is $q^{n-1}$-1
and hence $\digamma(\mathbb{V})$ is a
\emph{(}$q^{n-1}$-1\emph{)}-regular graph.
\end{lemma}

\begin{proof} Assume
  $u\in \mathbb{V}$ is an arbitrary element.
Then linear functional $f_u:\mathbb{F}^n_{q}\rightarrow
\mathbb{F}_q$ is an $\mathbb{F}_q$-module epimorphism. Hence the
quotient module $\frac{\mathbb{F}_q^n}{ker(f_u)}$ and $\mathbb{F}_q$
are isomorphic as $\mathbb{F}_q$-modules.
 So $ker(f_u)$ contains  exactly $q^{n-1}-1$ nonzero elements.
  Thus $deg_{_{\digamma(\mathbb{V})}}(f_u)$$=$$q^{n-1}-1$. Then for every $v\in \mathbb{V}$,
 clearly $v\sim f_u$ if and only if $f_v\sim u$.
  Hence
  $q^{n-1}-1=deg_{_{\digamma(\mathbb{V})}}(f_u)=deg_{_{\digamma(\mathbb{V})}}(u)$.
So the degree of every vertex of $\digamma(\mathbb{V})$ is
$q^{n-1}$-1 and hence $\digamma(\mathbb{V})$ is a
($q^{n-1}$-1)-regular graph.
\end{proof}
\begin{lemma}\label{a2} The sets
 $\mathbb{V}$ and $\mathbb{T}$ have the same domination
numbers in $\digamma(\mathbb{V})$ and it is $q+1$.
\end{lemma}
\begin{proof} The proof is similar as that of {\cite[Lemma
3.1]{XWang1}}. By Lemma \ref{regular} for any $u\in \mathbb{V}$,
$deg_{_{\digamma(\mathbb{V})}}(f_u)=q^{n-1}-1$. Note that for every
$v\in \mathbb{V}$, $ker(f_u)$ is a hyperspace of $\mathbb{V}_0$.
 Then clearly for every two distinct elements $u,
v\in \mathbb{V} $, $ker(f_u)\cap ker(f_v)$ contains at least a
$(n-2)$-dimensional  subspace of $\mathbb{V}_0$. So $ker(f_u)\cap
ker(f_v)$ has at least $q^{n-2}-1$ nonzero elements. Suppose $\{f_1,
f_2,\ldots,f_s\}$ is a dominating set of $\mathbb{\mathbb{V}}$ in
$\digamma(\mathbb{V})$. Denote the neighbor set of $f_i$ in
${\digamma(\mathbb{V})}$ by $N(f_i)$.
 So $|\bigcup_{i=1}^sN(f_i)|\leq s(q^{n-1}-1)-(q^{n-2}-1)$.
Hence $s(q^{n-1}-1)-(q^{n-2}-1)\geq (q^n-1)$. So $s>q$ and hence
$s\geq q+1$. On the other hand we will show that the set
$\{f_{e_1+ae_2}\mid a\in \mathbb{F}_q\}\cup \{f_{e_2}\}$ is a
dominating set for $\mathbb{\mathbb{V}}$ in $\digamma(\mathbb{V})$,
with $q+1$ elements, where $e_1=(1,0,\ldots,0)$ and
$e_2=(0,1,0\ldots,0)$. To do this, let $v=\sum_{i=1}^na_ie_i\in V$.
Then if $a_2=0$ then $f_{e_2}\sim v$ and if $a_2\neq 0$ then
$f_{e_1-\frac{a_1}{a_2}e_2}\sim v$. Therefore the domination number
of $\mathbb{\mathbb{V}}$ in $\digamma(\mathbb{V})$ is $q+1$. Now,
since by Lemma \ref{regular}, $\digamma(\mathbb{V})$ is a
$(q^n-1)$-regular graph, $|\mathbb{T}|=|\mathbb{V}|$ and $f_u\sim v$
if and only if $f_v\sim u$ for every $u,v\in \mathbb{V}$, one can
prove that the domination number of $\mathbb{T}$ is also $q+1$,
similar as that of $\mathbb{V}$.
\end{proof}
\begin{theorem}\label{a4}
The domination number of $\digamma(\mathbb{V})$ is $2q+2$.
\end{theorem}
\begin{proof}
Let $D$ be a dominating set of $V=\mathbb{T}\cup \mathbb{V}$ in
$\digamma(\mathbb{V})$. Clearly $D\cap \mathbb{T}$ and $D\cap
\mathbb{V}$ are domination sets for  $\mathbb{V}$ and $\mathbb{T}$
in $\digamma(\mathbb{V})$, respectively. So by Lemma \ref{a2}, $\mid
D\cap \mathbb{T}\mid \geq q+1$ and $\mid D\cap \mathbb{V}\mid\geq
q+1$. Then  $(D\cap \mathbb{T})\cap (D\cap \mathbb{V})=\emptyset$
implies $\mid D\mid \geq 2q+2$. Note that by Lemma \ref{a2} there
are minimal dominating sets $D_1$ and $D_2$ for $\mathbb{T}$ and
$\mathbb{V}$ in $\digamma(\mathbb{V})$, respectively, such that
$\mid D_1\mid = \mid D_2\mid=q+1$. Clearly $D_1\cup D_2$ is a
dominating set for $V=\mathbb{T}\cap \mathbb{V}$ in
$\digamma(\mathbb{V})$ and  $|D_1\cup D_2|=2q+2$. So the domination
number of $\digamma(\mathbb{V})$ is $2q+2$.
\end{proof}
This section ends with the following remark about the connectivity
of $\digamma(\mathbb{V})$.
\begin{remark} \label{connactivity} Assume $n\geq 3$. Then for all
$x, y \in \mathbb{V}$, $ker(f_x)\cap ker(f_y)$ is nonzero, as
$ker(f_x)$ and  $ker(f_y)$ are hyperspaces  with dimensions grater
than 1. Then the graph $\digamma(\mathbb{V})$ is a connected graph.
To see this, note that for any $x,y \in \mathbb{V}$ such that
$f_x\neq f_y$, we have $f_x\sim y$ (and hence $x\sim f_y$)
otherwise, there exist two paths $f_x\sim a\sim f_y\sim t\sim f_a
\sim y$ and $x\sim f_a\sim y$, for a nonzero $a\in ker(f_x)\cap
ker(f_y)$ and a nonzero $t\in ker(f_y)\cap ker(f_a)$. For $n=2$ in
Remark \ref{n=2} it is shown that $\digamma(\mathbb{V})$ is a
disconnected graph.
\end{remark}
\section{\bf Twin points and the structure of automorphisms of $\digamma(\mathbb{V})$}

 For a simple graph $\digamma$$=$$(V, E)$ and any nonempty subset $A$
of $V$, denote the neighbor of the set $A$ by $N(A)=\{x\in V\mid x
\sim a$ for some $a\in A\}$. For the simplicity if $A=\{a\}$, we
denote $N(A)$ by $N(a)$. The vertices $x$ and $y$ of $\digamma$ are
said to be \emph{twin points} if $x$ and $y$ have the same neighbors
(i.e. $N(x) = N(y)$). Consider a binary relation $\mathcal{R} $ in
$V$ as: $x\mathcal{R}y$ if and only if $N(x)=N(y)$, (i.e.
$x\mathcal{R}y$ if and only if $x$ and $y$ are twin points). It is
known that $\mathcal{R} $ is an equivalence relation in $V$.  Let
$\tau$ be a mapping on the vertex set $V$ of $\digamma$, which
stabilizes every equivalent class of $V$ and acts as a permutation
on every equivalent class. It is easy to see that $\tau(x) \sim
\tau(y)$ if and only if $x \sim y$. Thus $\tau$ is an automorphism
of the graph $\digamma$, which is called a \emph{permutation of twin
points} of $\digamma$. In this section, twin points classes of the
graph $\digamma(\mathbb{V})$ are characterized. Then, as we will see
in the next, the classes of twin points play main roles in
characterizing the automorphisms and determining the cardinal number
of automorphism group of $\digamma (\mathbb{V})$.

\begin{lemma}\label{a5} The vectors $u, v\in\mathbb{V}$ are twin points of  the graph $\digamma(\mathbb{V})$ if
and only if $u$ and $v$ are linearly dependent vectors.
\end{lemma}

\begin{proof} Let $u$ and $v$ have the same neighbors. So
 two linear equations
$u^Tx=0$ and $v^Tx=0$ have the same solutions in the vector space
$\mathbb{V}_0$. Therefore the coefficient matrices $u^T$ and $v^T$
are row-equivalent. This implies $u=rv$ for a nonzero scaler $0\neq
r\in \mathbb{F}_q$. So $u$ and $v$ are linearly dependent. To the
converse let $u=rv$ for a scaler $0\neq r\in \mathbb{F}_q$. Then
$u^T$ and $v^T$ are row-equivalent matrices. So the linear equations
systems $u^Tx=0$ and $v^Tx=0$ have the same solutions. So
$ker(f_u)=ker(f_v)$, hence $N(u)=N(v)$ and the proof is complete.
\end{proof}
\begin{remark}\label{cardSigma} In this paper $\Sigma_0$ denotes
 the set of all 1-dimensional subspaces of
$\mathbb{V}_0=F^n_q$ and $\Sigma=\{W \setminus\{0\}\mid\, W
\in\Sigma_0\}$ the set of all 1-dimensional subspaces without zero
vector of $\mathbb{V}_0$. In \cite{XWang1} the authors showed that
   $\Sigma$ totally contains $\frac{q^n-1}{q-1}$
elements and every $S\in\Sigma $ consists of exactly $q-1$ nonzero
vectors. Also $\mathbb{V}$ is disjoint union of all $S$ for $S\in
\Sigma$. Along this paper for any $A\subseteq \mathbb{V}$  we use
the notation $\mathcal{F}_A=\{f_u| u\in A\}$. It is easy to see that
$\mathcal{F}_S$ is a 1-dimensional subspace without zero of
$\mathbb{T}_0$, for every $S\in \Sigma$ .
\end{remark}
  Now, using Lemma $\ref{a5}$ we have the following Result.
\begin{lemma}\label{a6}\par
\indent \emph{\textbf{(1)}} The vectors $u,v\in \mathbb{V}$ as the
vertices of $\digamma(\mathbb{V})$ are twin points if and only if
there exists an element $S\in\Sigma$ such
that $u,v \in S$.\\
\indent\emph{\textbf{(2)}}Two vertices $f_u,f_v\in \mathbb{T}$ are
twin points if and only if
 $u$ and $v$ are  twin points.\\
\indent\emph{\textbf{(3)}}For any $S\in\Sigma$,  $a\in S$ if and
only if
$N(S)$$=$$N(a)$ if and only if $N(\mathcal{F}_S)$$=$$N(f_a)$.\\
\indent\emph{\textbf{(4)}}For every $S_1, S_2\in \Sigma$,  $S_1\neq
S_2$ if
and only if  $N(S_1)\neq N(S_2)$.\\
\indent\emph{\textbf{(5)}}For every $S\in \Sigma$, the subgraph
$\langle S\cup N(S)\rangle$ of $\digamma(\mathbb{V})$,
    is a complete bipartite graph isomorphic with $\mathcal{K}_{q-1,q^{n-1}-1}$.
\end{lemma}
\begin{proof}  Let $S\in \Sigma$. Then every two elements $u$ and $v$
 of $S$ are linearly
dependent. So by Lemma \ref{a5}, all elements of $S$ have the same
neighbors and using Lemma \ref{a5} one can prove all the parts.
\end{proof}

Now, the aim is characterizing the structure of all automorphisms of
the graph $\digamma(\mathbb{V})$. First we point to the two needed
especial kinds of automorphisms for this graph  below:\par
\textbf{(1)}\indent Similar as \cite{XWang1}  for any invertible
$n\times n$ matrix $P$ over $\mathbb{F}_q$, let $\chi_{_P}$  be the
mapping on the vertex set of $\digamma(\mathbb{V})$ such that:
\begin{center}
$\chi_{_P}(v) = Pv$ for $v\in\mathbb{V}$ and $\chi_{_P}(f_u) =
f_{(P^{-1})^Tu}$ for any $u\in \mathbb{V}$.
\end{center}
 Assume  $u ,v \in \mathbb{V}$. Then,
$\chi_{_P}(f_u)(\chi_{_P}(v))=f_{(P^{-1})^Tu}(Pv)$
$=u^TP^{-1}(Pv)=u^Tv=f_u(v)$.  Hence $f_u(v)=0$ if and only if
$\chi_{_P}(f_u)(\chi_{_P}(v))=0$. So $\chi_{_P}$ is an automorphism
of the graph $\digamma(\mathbb{V})$, which we said to be the \emph{
regular automorphism} of $\digamma(\mathbb{V})$ induced by matrix
$P$.\par
 \textbf{(2)}\indent For any field automorphism $\pi$ of
$\mathbb{F}_q$ extend $\pi$ to vertices of $\digamma(\mathbb{V})$ as
follows:\par
 For every $v=\sum_{i=1}^nv_ie_i\in \mathbb{V}$, define
  $\pi (v)=\sum_{i=1}^n\pi(v_i)e_i$
 and for every $u\in \mathbb{V}$ define $\pi(f_u)=f_{\pi(u)}$.
Then one can see that this extending of $\pi$ on the vertices of
$\digamma (\mathbb{V})$ is an automorphism of the graph $\digamma
(\mathbb{V})$, which we said to be the\emph{ extending of field
automorphism} $\pi$ of the graph $\digamma(\mathbb{V})$.
\begin{remark} \label{cardsigma}We define  for the set $\Sigma$, the set $N(\Sigma)=\{N(S)\mid
S\in \Sigma\}$. So by Lemma \ref{a6} part 4 and Remark
\ref{cardSigma}, $|\Sigma|=|N(\Sigma)|=\frac{q^n-1}{q-1}$.
 \end{remark}
\begin{lemma}\label{a7} Let $\rho:V\rightarrow V$ be a permutation on the vertex set of
$\digamma(\mathbb{V})$ such that  $(\rho(H), \rho(N(H)))\in
\{(S,N(S))|\, S\in \Sigma\}$, for all $H\in \Sigma$. Then $\rho$ is
an automorphism of the graph $\digamma(\mathbb{V})$.
\end{lemma}
\begin{proof} Note that the assumption implies
$N(\rho(H))=\rho(N(H))$, for all $H\in \Sigma$. Let $f_u\sim v$ for
some $u,v\in \mathbb{V}$. We claim that $\rho(f_u)\sim\rho(v)$.
Since $f_u(v)=0$, there exist a unique $H\in\Sigma $ such that $v\in
H$ and $f_u\in N(H)$. By Lemma \ref{a6} for any $S\in \Sigma$, the
subgraph  $\langle S\cup N(S)\rangle$ is isomorphic with the
complete bipartite graph $\mathcal{K}_{q-1, q^{n-1}-1}$. By
assumption, $(\rho(H), \rho(N(H)))\in \{(S,N(S))|\, S\in \Sigma\}$.
Hence there exists an element $S\in \Sigma$ such that $\rho(H)=S$
and $\rho(N(H))=N(S)$. Hence $\rho (v)\in S$ and $\rho(f_u)\in
N(S)$. So $\rho(f_u)(\rho(v))=0$ and the claim is proven.
Conversely, let $\rho(f_u)(\rho(v))=0$, for two elements $u,v\in
\mathbb{V}$. Then there exists a unique $H\in \Sigma$ such that
$v\in H$. Let $\rho(H)=S$ for a $S\in \Sigma$. Then
 $\rho(f_u)(\rho(v))=0$ implies $\rho(f_u)\in N(\rho(H))=\rho(N(H))=N(S)$. Hence
  $\rho(f_u)=\rho(f_{u'})$, for
 a $f_{u'}\in N(H)$.
 Since $\rho$ is one to one correspondence, we have $f_u=f_{u'}\in
 N(H)$. So $f_u(v)=0$ as the subgraph $\langle H\cup N(H)\rangle$ of
 $\digamma(\mathbb{V})$ is a
complete bipartite graph. Hence $\rho$ is an automorphism of $\digamma(\mathbb{V})$.\\
\end{proof}

\begin{corollary}
Let $\rho$ be a permutation map on the vertex set of
$\digamma(\mathbb{V})$ such that  $\rho(H)$$=$$H$ and
$\rho(N(H))$$=$$N(H)$, for all $H\in\Sigma$. Then $\rho$ is an
automorphism of $\digamma(\mathbb{V})$.
\end{corollary}
\begin{definition} Let $\rho :V\rightarrow V$ be a permutation map on the vertex set of
$\digamma(\mathbb{V})$ such that  $(\rho(H), \rho(N(H)))\in
\{(S,N(S))|\, S\in \Sigma\}$ for all $H\in \Sigma$. Then we say
$\rho$ is a  \emph{one-dimensional subspaces permutation
automorphism} of graph $\digamma(\mathbb{V})$. In particular when
$\rho(H)=H$ and $\rho(N(H))=N(H)$, for all $H\in \Sigma$, then
$\rho$ is said to be a \emph{twin points automorphism} of the graph
$\digamma(\mathbb{V})$.

\end{definition}
\begin{lemma}\label{4partition}
Let $\rho$ be  an automorphism of the graph
$\digamma(\mathbb{V})$. Then:\\
\indent \emph{\textbf{(1)}} If $\rho (X)\subseteq \mathbb{T}$ for
some
  non empty $X\subseteq \mathbb{V}$, then
  $\rho( N(X))\subseteq \mathbb{V}$.\\
 \indent \emph{\textbf{(2)}} If $\rho (Y)\subseteq \mathbb{V}$ for some
  non empty $Y\subseteq \mathbb{T}$, then $\rho( N(Y))\subseteq \mathbb{T}$.\\
\indent \emph{\textbf{(3)}} For any $S\in \Sigma$ either
$\rho(N(S))\subseteq
\mathbb{V}$ or $\rho(N(S))\subseteq \mathbb{T}$. \\
\indent \emph{\textbf{(4)}} For any $S\in \Sigma$ either
$\rho(S)\subseteq \mathbb{V}$ or $\rho(S)\subseteq \mathbb{T}$.
 \end{lemma}
\begin{proof}
\textbf{(1)} Since $\rho$ is an automorphism, then every element of
$\rho(N(X))$ is adjacent to an element $\rho(X)\subseteq
\mathbb{T}$. Hence $\rho( N(X))\subseteq \mathbb{V}$.\par
\indent \textbf{(2)} The proof is similar to that of the part 1.\\
\indent \textbf{(3)} By Lemma \ref{a6} for any $x\in S$,
$N(x)=N(S)$. Now, either $\rho(x)\in \mathbb{V}$ or $\rho (x)\in
\mathbb{T}$. Hence by parts 1 and 2 either $\rho(N(x))\subseteq
\mathbb{T}$ or $\rho(N(x))\subseteq \mathbb{V}$. Therefore either
$\rho(N(S))\subseteq \mathbb{T}$ or
$\rho(N(S))\subseteq \mathbb{V}$.\\
\indent \textbf{(4)} The proof is similar to that of the part 3.
\end{proof}

\begin{remark}\label{Nandrho} Let $\rho$  be a
 permutation on the vertex set $V$ of a graph $\digamma$.
  Then  $\rho$ is an automorphism
   of the graph $\digamma$,
  if and only if for every $A\subseteq V$, $\rho(N(A))=N(\rho(A))$.
  To the prove, let $\rho$ be an automorphism, then
  $x\in \rho(N(A))\,\Leftrightarrow\rho^{-1}(x)\in N(A)$
$\Leftrightarrow(\rho^{-1}(x)\sim
  a$, for some $a\in A$) $\Leftrightarrow( x\sim \rho(a)$, for some
   $a\in A)\Leftrightarrow x\in N(\rho(A))$. Conversely,
   the assumption implies, for any $x\in V$, $\rho(N(x))=N(\rho(x))$. So
    $a\sim x\Leftrightarrow a\in N(x)\,\Leftrightarrow \rho(a)\in \rho(N(x))$
$\Leftrightarrow \rho(a)\in N(\rho(x))\Leftrightarrow \rho(a)\sim
\rho(x)$.
   \end{remark}
\begin{lemma}\label{a8}
Let $\rho$ be an automorphism of the graph
$\digamma(\mathbb{V})$ and $u,v\in \mathbb{V}$. Then:\\
\indent \emph{\textbf{(1)}} If $\rho(u),\rho(v)\in \mathbb{V}$, then
$\rho(u)$ and $\rho(v)$ are linearly independent, if and only if $u$
and $v$
are linearly independent vectors  of $\mathbb{V}_0$. \\
\indent \emph{\textbf{(2)}} If $\rho(u),\rho(v)\in \mathbb{T}$, then
$\rho(u)$ and $\rho(v)$ are linearly independent, if and only if $u$
and $v$ are linearly independent vectors  of $\mathbb{V}_0$.
\end{lemma}
\begin{proof}
(1) Let $u$ and $v$ are linearly independent. By Lemma \ref{a5} $u$
and $v$ are not twin points. To the contrary, assume $\rho(u)$ and
$\rho(v)$ are linearly dependent. Then by Lemma \ref{a6}
$N(\rho(u))=N(\rho(v))$. Then since $\rho$ is automorphism by Remark
\ref{Nandrho}, $\rho(N(x))=N(\rho(x))$ for every vertex $x$ of
$\digamma(\mathbb{V})$. So $\rho(N(u))= \rho(N(v))$ and hence
$N(u)=N(v)$, as $\rho$ is one to one correspondence.  Therefore $u$
and $v$ are twin points, a contradiction. So $\rho(u)$ and $\rho(v)$
are linearly independent. The converse is similar.\\
(2) The proof is similar as that of the part 1.
\end{proof}

\begin{lemma}\label{a9}
Let $\rho$ be an automorphism of the graph $\digamma(\mathbb{V})$.
Then for any $S\in \Sigma$, either $\rho (S)\in \Sigma$ or
$\rho(S)=\mathcal{F}_{S'}$  for some  $S'\in \Sigma$. In other words
$\rho$ acts as a permutation on the set $\Sigma \cup
\{\mathcal{F}_S|\, S\in \Sigma\}$.
\end{lemma}
\begin{proof} Notice that for any $S'\in \Sigma$, $\mathcal{F}_{S'}$
 is a 1-dimensional subspace without zero vector of
$\mathbb{T}_0$. Let $\{x, y\}\subseteq \mathbb{V}$ or $\{x,
y\}\subseteq \mathbb{T}$. Then the Corollary \ref{a8} implies,
$\rho(x)$ and $\rho(y)$ are linearly dependent if and only if $x$
and $y$ are linearly dependent vectors in $\mathbb{V}$ or in
$\mathbb{T}$. Note that two vertices of $\digamma(\mathbb{V})$ are
linearly dependent vectors if and only if they are elements of a
unique 1-dimensional subspace without zero say $S$ or
$\mathcal{F}_S$, for a $S\in \Sigma$. Therefore $\rho$ acts as a
permutation on the set $\Sigma \cup \{\mathcal{F}_S|\, S\in
\Sigma\}$, as $S$ and $\mathcal{F}_S$ are 1-dimensional subspaces
without zero vector of $\mathbb{V}_0$ and $\mathbb{T}_0$,
respectively, for any $S\in \Sigma$.
\end{proof}

Now, the following result is useful for characterizing the
automorphisms of the graph $\digamma(\mathbb{V})$ in the next
results.

\begin{theorem}\label{maintwo}
Let $\rho: V\rightarrow V$ be bijective map on the vertex set $V=
\mathbb{V}\cup \mathbb{T}$ of the graph $\digamma(\mathbb{V})$. Then
the following conditions are
equivalent:\\
\indent\emph{\textbf{(1)}} $\rho$ is an automorphism of the graph $\digamma(\mathbb{V})$.\\
\indent\emph{\textbf{(2)}} $\rho$ acts as a permutation on $\Sigma
\cup\{\mathcal{F}_S|\, S\in \Sigma\}$ and  $\rho(N(S))=N(\rho (S))$
for every $S\in \Sigma$.\\
\indent\emph{\textbf{(3)}} $\rho$ acts as a permutation on $\Sigma
\cup\{\mathcal{F}_S|\, S\in \Sigma\}$ and
$\rho(N(\mathcal{F}_S))=N(\rho (\mathcal{F}_S))$ for every $S\in
\Sigma$.
\end{theorem}
\begin{proof}
$\textbf{(1)}\Rightarrow \textbf{(2)}$ Let $\rho$ be an automorphism
of
 $\digamma(\mathbb{V})$, then by  Lemma \ref{a9}, $\rho$ acts as a permutation on the set $\Sigma
\cup\{\mathcal{F}_S|\, S\in \Sigma\}$. By Remark \ref{Nandrho}
$\rho(N(S))=N(\rho (S))$, for all $S\in \Sigma$.\par
 $\textbf{(2)}\Rightarrow \textbf{(1)}$ Assume $\rho$ satisfies the condition 2 and $x, s \in
\mathbb{V}$. Then there exists a unique  $S\in \Sigma$ such that
$s\in S$ and by Lemma \ref{a6}  $N(s)=N(S)$. So $f_x\sim s$ if and
only if $f_x\in N(s)=N(S)$, if and only if $\rho(f_x)\in
\rho(N(s))=\rho(N(S))=N(\rho(S))$, if and only if
$\rho(f_x)\sim\rho(s')$ for some $s'\in S$.  By assumption
$\rho(S)\in \Sigma\cup \{\mathcal{F}_S|\, S\in \Sigma\}$. So
$\rho(S)= H$ or $\rho(S)=\mathcal{F}_{H}$, for a unique $H\in
\Sigma$. Since $\rho(s')\in H$ or $\rho(s')\in\mathcal{F}_{H}$, by
Lemma \ref{a6} $N(\rho(s'))= N(H)$ or
$N(\rho(s'))=N(\mathcal{F}_{H})$, for some $H\in \Sigma$.  In both
cases $N(\rho(s'))=N(\rho(S))$. So $\rho(f_x)\sim\rho(s')$ for some
$s'\in S$, if and only if $\rho(f_x)\in N(\rho(S))$, if and only if
$\rho(f_x)\sim\rho(s)$, as $\rho(S)\in \Sigma \cup\{\mathcal{F}_S|\,
S\in \Sigma\}$ and hence $\langle\rho(S)\cup N(\rho(S))\rangle$ is a
complete bipartite graph. So $f_x\sim s$ if and only if
$\rho(f_x)\sim\rho(s)$.\par
 \indent$\textbf{(3)}\Leftrightarrow\textbf{(1)}$ The proof
of this equivalency is similar as that of $(2)\Leftrightarrow(1)$.
\end{proof}
The following Lemma is useful for the next.
\begin{lemma}\label{intersection} Let $\rho$ be an automorphism of the graph
   $\digamma (\mathbb{V})$.
   Then, for every $H\in \Sigma$,
   $\rho (\mathcal{F}_H)=\bigcap_{\mathcal{F}_H\subseteq
   N(S)}N(\rho(S))$.
 \end{lemma}
\begin{proof}
If $\mathcal{F}_H\subseteq N(S)$ for a $S\in \Sigma$, then $\rho
(\mathcal{F}_H)\subseteq \rho(N(S))=N(\rho(S))$, as by Theorem
\ref{maintwo} $\rho(N(S))=N(\rho(S))$, for every $S\in \Sigma$.  So
$\rho (\mathcal{F}_H)\subseteq\bigcap_{\mathcal{F}_H\subseteq
N(S)}N(\rho(S))$. Now assume $x\in\bigcap_{\mathcal{F}_H\subseteq
N(S)}N(\rho(S))$. For every $S\in \Sigma$, if
$\mathcal{F}_H\subseteq N(S)$ then $x\in N(\rho(S))=\rho(N(S))$.
Therefore for every $S\in \Sigma$, if $\mathcal{F}_H\subseteq N(S)$
then $\rho^{-1}(x)\in N(S)$. Note that every hyperspace without zero
element of $\mathbb{T}_0$ is a form of $N(S)$, for a unique $S\in
\Sigma$. So by {\cite[Corollary after Theorem 16]{Hoffman}}
$\mathcal{F}_H$ is intersection of number of $(n-1)$ hyperspaces
without zero of $\mathbb{T}_0$ and every such a hyperspace without
zero element is a $N(S)$ containing $\mathcal{F}_H$, for a unique
$S\in \Sigma$. Therefore $\mathcal{F}_H=\bigcap
_{\mathcal{F}_H\subseteq N(S)} N(S)$. Thus $\rho^{-1}(x)\in
\mathcal{F}_H$. So $x\in \rho (\mathcal{F}_H)$. Hence $\rho
(\mathcal{F}_H)\supseteq\bigcap_{\mathcal{F}_H\subseteq
N(S)}N(\rho(S))$. So
$\rho(\mathcal{F}_H)=\bigcap_{\mathcal{F}_H\subseteq
N(S)}N(\rho(S))$.
\end{proof}
Here, first we start investigate the structure of automorphisms of
the graph $\digamma(\mathbb{V})$, for  $n$=2. Since we will see in
Remark \ref{n=2}, for $n=2$ the graph $\digamma(\mathbb{V})$ is not
connected graph. Also we will see the main theorems for the cases
$n=2$ and $n\geq 3$ are not similar.

\begin{remark} \label{n=2} Let $n=2$ and $\mathbb{V}_0=\mathbb{F}_q^2$.
Then $W$ is a hyperspace of $\mathbb{V}_0$ or $\mathbb{T}_0$ if and
only if $W$ is a 1-dimensional subspace of $\mathbb{V}_0$ or
$\mathbb{T}_0$, respectively. As we know for every $S\in \Sigma$,
$N(S)$ and $N(\mathcal{F}_S)$ are hyperspaces without zero elements
of $\mathbb{T}_0$ and $\mathbb{V}_0$, respectively.
 Hence $N(S)\in \{\mathcal{F}_H| H\in \Sigma\}$ and
$N(\mathcal{F}_S)\in \Sigma $, for all $S\in \Sigma$. By Remark
\ref{cardsigma}, $|\Sigma|=\frac{q^2-1}{q-1}=q+1$. Let
$\Sigma=\{S_1,S_2,\ldots,S_{q+1}\}$. By Lemma \ref{a6}, if $S_i\neq
S_j$ then $N(S_i)\cap N(S_j)=\emptyset$, as the intersection of two
distinct 1-dimensional subspaces is zero.
 Therefore $\digamma(\mathbb{V})$ is a disconnected graph with
$q+1$ connected components $\langle S_1\cup N(S_1)\rangle$,$\langle
S_2\cup N(S_2)\rangle$,$\ldots$,$\langle S_{q+1}\cup
N(S_{q+1}\rangle$. Clearly $\langle S_i\cup N(S_i)\rangle$ is a
complete bipartite graph isomorphic with $\mathcal{K}_{{q-1},{q-1}}$
for all $i=1,2,\ldots,q+1$. So every two components are isomorphic
graphs.
 \end{remark}
\begin{lemma}\label{Lemn=2} Assume $n=2$ and $S_1,S_2\in \Sigma$.
Let $\sigma:S_1\cup N(S_1)\rightarrow S_2\cup N(S_2)$ be a
 bijection. Then the following conditions hold:\\
 \indent \emph{\textbf{(1)}}  $\sigma$ is a graph isomorphism  from $\langle S_1\cup
 N(S_1)\rangle$ to $\langle S_2\cup
 N(S_2)\rangle$ if and only if  either $\sigma(S_1)=S_2$ or
$\sigma(S_1)=N(S_2)$.\\
 \indent \emph{\textbf{(2)}} There are exactly
 $2((q-1)!)^2$ graph isomorphisms from $\langle S_1\cup
 N(S_1)\rangle$ to $\langle S_2\cup
 N(S_2)\rangle$.
 \end{lemma}
\begin{proof} \textbf{(1)} Clearly if $\sigma(S_1)=S_2$ (i.e. $\sigma(N(S_1))=N(S_2)$ ) or
 $\sigma(S_1)=N(S_2)$ (i.e. $\sigma(N(S_2))=S_1$ ),  then $\sigma$ is an
 isomorphism. To the converse, similar as the proof of Lemma \ref{a8},
 $\sigma$ must send every two linearly dependent vectors to two linearly dependent vectors.
 Since $|N(S_i)|=|S_i|=q-1$, we have $\sigma(S_1)=S_2$ or
 $\sigma(S_1)=N(S_2)$.\\
\indent \textbf{(2)} There are exactly the number of $(q-1)!(q-1)!$
bijections $\sigma:S_1\cup N(S_1) \rightarrow S_2\cup N(S_2)$  such
that $\sigma(S_1)=S_2$ and $\sigma(N(S_1))=N(S_2)$. Also there are
exactly the number of $(q-1)!(q-1)!$ bijections $\sigma:S_1\cup
N(S_1) \rightarrow S_2\cup N(S_2)$ such that $\sigma(S_1)=N(S_2)$
and $\sigma(N(S_1))=S_2$. So the number of graph isomorphisms such a
$\sigma:\langle S_1\cup
 N(S_1)\rangle \rightarrow\langle S_2\cup
 N(S_2)\rangle$ is $2((q-1)!)^2$.
\end{proof}
In the following result the structure of automorphisms of
$\digamma(\mathbb{V})$, where $\mathbb{V}_0=\mathbb{F}_q^2$ is
characterized.
 \begin{theorem}\label{thn=2}Let $n=2$. Then, a permutation $\rho$ of the set
 $V=\mathbb{V}\cup \mathbb{T}$ is an automorphism of the graph
 $\digamma(\mathbb{V})$, if and only if two following conditions
 hold:\\
\indent \emph{\textbf{(1)}} $\rho$ acts as a permutation on the set
of components\emph{(}i,e. for every $S_i\in \Sigma$  there exists a
unique $S_j\in \Sigma$ such that $\rho (S_i\cup N(S_i))= S_j\cup
N(S_j)$.\\
\indent \emph{\textbf{(2)}} For $S_i,S_j\in \Sigma$, if $\rho
(S_i\cup N(S_i))= S_j\cup N(S_j)$ then $\rho (S_i)\in \{S_j,
N(S_j)\}$.
 \end{theorem}
 \begin{proof} Let $\rho$ be an automorphism, then using Remark \ref{n=2}, clearly $\rho$
 sends every connected components to a connected component. So
 $\rho$ acts as a permutation on the set $\{S_i\cup N(S_i)| S_i\in
 \Sigma\}$. Hence the condition 1 holds. Assume $\rho(S_i\cup
 N(S_i))=S_j\cup N(S_j)$.  Then $\rho$ acts necessarily as a graph isomorphism between
 $\langle S_i\cup N(S_i)\rangle$ and $\langle S_j\cup
 N(S_j)\rangle$. So by Lemma \ref{Lemn=2}, we have $\rho(S_i)=S_j$ or
$\rho(S_i)=N(S_j)$ and hence the condition 2 holds. The converse is
clear as each component of $\digamma(\mathbb{V})$ is isomorphic to
the complete bipartite graph $\mathcal{K}_{q-1,q-1}$ and the proof
is complete.
 \end{proof}
 By Theorem \ref{thn=2}, for $n=2$ it is possible to have an
 automorphism $\rho$ of $\digamma(\mathbb{V})$ such that
 $\rho(\mathbb{V})\notin \{\mathbb{V}, \mathbb{T}\}$. But for $n\geq
 3$ the following holds.
\begin{corollary}\label{V or T} Let $\rho$ be an automorphism of
$\digamma (\mathbb{V})$ and $n\geq3$.
   Then, either $\rho (\mathbb{V})=\mathbb{V}$ or $\rho
   (\mathbb{V})=\mathbb{T}$\emph{(}i.e. either $\rho (\mathbb{T})=\mathbb{T}$ or $\rho
   (\mathbb{T})=\mathbb{V}$\emph{)}.
\end{corollary}
\begin{proof}
Note that by Lemma \ref{4partition}, for any $S\in\Sigma$  either
$\rho(S)\subseteq \mathbb{V}$ or $\rho(S)\subseteq \mathbb{T}$. To
the contrary assume $S_1\neq S_2$ are two elements of $\Sigma$ such
that $\rho(S_1)\subseteq \mathbb{V}$ and
$\rho(S_2)\subseteq\mathbb{T}$. Then consider $\mathcal{F}_H$, for a
$H\in \Sigma$, such that $\mathcal{F}_H\subseteq N(S_1)\cap
N(S_2)$,(Note that $n\geq 3$ and hence $N(S_1)\cap N(S_2)\neq
\emptyset$
 and hence there exists such a $\mathcal{F}_H$).
Then by Lemma \ref{intersection},
$\rho(\mathcal{F}_H)=\bigcap_{\mathcal{F}_H\subseteq
N(S)}N(\rho(S))\subseteq N(\rho(S_1))\cap N(\rho(S_2))=\emptyset$,
as $N(\rho(S_1))\subseteq \mathbb{T}$ and $N(\rho(S_2))\subseteq
\mathbb{V}$, a contradiction. So either $\rho(S)\subseteq
\mathbb{V}$ for all $S\in \Sigma$, or $\rho(S)\subseteq \mathbb{T}$
for all $S\in \Sigma$. This implies either $\rho
(\mathbb{V})=\mathbb{V}$ or $\rho
   (\mathbb{V})=\mathbb{T}$.
\end{proof}
In the  following result for $n\geq 3$, the structure of
automorphisms of the graph $\digamma(\mathbb{V})$ is characterized.
\begin{theorem}\label{mainth2} For the integer $n\geq 3$
 the permutation $\rho:V\rightarrow V$ is an automorphism of
$\digamma(\mathbb{V})$, if and only if exactly one of the following
conditions holds:\\
\indent\emph{\textbf{(1)}} $\rho$ acts as a permutation on $\Sigma$
and $\rho
(\mathcal{F}_H)$$=$$\bigcap$$_{_{\mathcal{F}_H\subseteq N(S)}}$$N(\rho(S))$, for all $H\in \Sigma$.\\
\indent \emph{\textbf{(2)}} $\rho$ acts as a bijection
$\Sigma\rightarrow \{\mathcal{F}_H|\, H\in \Sigma\}$ and $\rho
(\mathcal{F}_H)=\bigcap_{\mathcal{F}_H\subseteq N(S)}N(\rho(S))$,
for all $H\in \Sigma$.
\end{theorem}
\begin{proof}
Let $\rho$ be an automorphism, then by Corollary \ref{V or T},
either $\rho(\mathbb{V})=\mathbb{V}$ or
$\rho(\mathbb{V})=\mathbb{T}$. By Theorem \ref{maintwo} $\rho$ acts
as a permutation on $\Sigma\cup\{\mathcal{F}_H|H\in\Sigma\}$, so
clearly at most one of the two conditions holds. So if
$\rho(\mathbb{V})=\mathbb{V}$, then by Lemma \ref{intersection},
$\rho (\mathcal{F}_H)=\bigcap_{\mathcal{F}_H\subseteq
N(S)}N(\rho(S))$, for all $H\in \Sigma$ and hence the condition 1
holds. If $\rho(\mathbb{V})=\mathbb{T}$, then  by Theorem
\ref{maintwo} $\rho$ acts as a one to one correspondence
$\Sigma\rightarrow \{\mathcal{F}_H|H\in\Sigma\}$ and hence the
condition (2) holds. To the converse, assume the condition (1)
holds. We have to show that $\rho$ is an automorphism of the graph
$\digamma(\mathbb{V})$. First note that $\mathcal{F}_H$ is a
1-dimensional subspace without zero element of $\mathbb{T}_0$ and
$N(S)$ is a hyperspace without zero of $\mathbb{T}_0$, for every $S,
H\in \Sigma$. For a fixed $H\in \Sigma$ one can see that for every
$S\in\Sigma$ either $\mathcal{F}_H \subseteq N(S)$ or $\mathcal{F}_H
\cap N(S)=\emptyset$.
 On the other hand for any $S\in \Sigma$ since $\rho(S)\in
\Sigma$, we conclude that $N(\rho(S))$ is also hyperspace without
zero of $\mathbb{T}_0$. If $S_1\neq S_2$ then by Lemma \ref{a6}
$N(S_1)$$\neq$$N(S_2)$ and also since  $\rho(S_1)$$\neq$$\rho(S_2)$
we have $N(\rho(S_1))$$\neq$$N(\rho(S_2))$ two distinct hyperspaces
with out zero of $\mathbb{T}_0$.
 So by {\cite[Corollary
after Theorem 16]{Hoffman}} the intersection
$\bigcap_{\mathcal{F}_H\subseteq N(S)}N(\rho(S))$ is the
intersection of exactly ($n$-1) numbers of distinct hyperspaces
without zero of $\mathbb{T}_0$, as  $\mathcal{F}_H$ is an
intersection of exactly ($n$-1) numbers of distinct hyperspaces with
out zero of the form $N(S)$, for some $S\in \Sigma$. Hence
$\bigcap_{\mathcal{F}_H\subseteq N(S)}N(\rho(S))$ must be a
1-dimensional subspace without zero of $\mathbb{T}_0$ say
$\mathcal{F}_{H'}$, for a $H'\in \Sigma$. Therefore
$\rho(\mathcal{F}_H)$$=$$\mathcal{F}_{H'}$$\in$$\{\mathcal{F}_H|\,
H\in\Sigma\}$. Now, we claim that $\rho$ acts as a permutation on
$\{\mathcal{F}_H|\, H\in\Sigma\}$. To the claim assume $H$$\neq$$H'$
are two arbitrary elements of $\Sigma$. Then there exist $S\neq
S'\in \Sigma$ such that $\mathcal{F}_H\cap N(S')$$=$$\emptyset$,
$\mathcal{F}_{H'}\cap N(S)$$=$$\emptyset$, $\mathcal{F}_H\subseteq
N(S)$, $\mathcal{F}_{H'}\subseteq N(S')$. So
$\rho(\mathcal{F}_{H'})$$\nsubseteq$$N(\rho(S))$ and
$\rho(\mathcal{F}_{H})\nsubseteq N(\rho(S'))$. But
$\rho(\mathcal{F}_{H})\subseteq N(\rho(S))$ and
$\rho(\mathcal{F}_{H'})\subseteq N(\rho(S'))$. So
$\bigcap_{\mathcal{F}_H\subseteq N(S)}N(\rho(S))\neq
\bigcap_{\mathcal{F}_{H'}\subseteq N(S)}N(\rho(S))$ and hence
$\rho(\mathcal{F}_H)$$\neq$$\rho(\mathcal{F}_{H'})$ and  the claim
is proven. By Theorem \ref{maintwo} it is enough to show that
  $\rho(N(S))= N(\rho(S))$ for all $S\in \Sigma$. Let $f_x\in
  \rho(N(S))$ and so $f_x=\rho(f_{x'})$ for a unique $f_{x'}\in N(S)$.
Then there exists a unique $H\in\Sigma$ such that
$f_{x'}\in\mathcal{F}_H $. Note that $\mathcal{F}_H\subseteq N(S)$,
as $\mathcal{F}_H$ is a 1-dimensional subspace without zero in
$N(S)$, generated by $f_{x'}$. Then the condition (1) implies
$f_x=\rho(f_{x'})\in \rho
(\mathcal{F}_H)=\bigcap_{\mathcal{F}_H\subseteq N(S)}N(\rho(S))$.
Therefore  $f_x\in N(\rho(S))$. So $\rho(N(S))\subseteq N(\rho(S))$.
One can see that $|\rho(N(S))|$$=$$|N(\rho(S))|$ and so
$\rho(N(S))$$=$$N(\rho(S))$. Thus $\rho$ is an automorphism of the
graph $\digamma(\mathbb{V})$.\par The equivalency related to the
condition (2) has a similar proof as that of the condition (1) and
the proof is complete.
\end{proof}
\section{\bf Cardinal number of the automorphism group of $\digamma(\mathbb{V})$}
It is well-known in the group theory, the cardinal number of any
finite group is very important  parameter for studying a group and
its subgroups. In this section the aim is determining the cardinal
number of the automorphism group for the graph
$\digamma(\mathbb{V})$.
\begin{theorem}\label{card2} For $n=2$, the graph $\digamma(\mathbb{V})$ has
exactly $(q+1)!(2((q-1)!)^2)^{q+1}$ automorphisms.
\end{theorem}
\begin{proof} By Remark \ref{n=2} and Theorem \ref{thn=2}, the set of
automorphisms of $\digamma(\mathbb{V})$ are exactly the set of
extensions of all permutations on the set of components $\{\langle
S_1\cup N(S_1)\rangle$,$\langle S_2\cup
N(S_2)\rangle$,$\ldots$,$\langle S_{q+1}\cup N(S_{q+1}\rangle\}$ to
the automorphisms of $\digamma(\mathbb{V})$. The set of components
has $(q+1)!$ permutations. Consider a fixed permutation on the set
of components. Then by Lemma \ref{Lemn=2} this permutation makes the
number of $2((q-1)!)^2$ isomorphisms from a component to its image
under this permutation. Since there are $q+1$ components,  using
multiple principle we have  $(2((q-1)!)^2)^{q+1}$ automorphisms of
$\digamma(\mathbb{V})$, for  this permutation. So in total there are
$(q+1)!(2((q-1)!)^2)^{q+1}$ automorphisms  of $\digamma(\mathbb{V})$
for all of permutations.
\end{proof}

\begin{remark} \label{a10} Assume $n\geq 3$ and $\rho$ is an automorphism
of the graph $\digamma(\mathbb{V})$. Then, by Corollary \ref{V or
T}, either $\rho(\mathbb{V})=\mathbb{V}$( and hence by Theorem
\ref{mainth2}, $\rho$ act as a permutation on $\Sigma$) or
$\rho(\mathbb{V})=\mathbb{T}$(and hence by Theorem \ref{mainth2},
$\rho$ acts as a bijection $\Sigma\rightarrow\{\mathcal{F}_H|\,H\in
\Sigma\}$). Then by Theorem \ref{mainth2} for every $H\in \Sigma$,
$\rho(\mathcal{F}_H)$ is obtained uniquely by the intersection
$\bigcap_{\mathcal{F}_H\subseteq N(S)}N(\rho(S))$. By Theorem
\ref{mainth2} the set of automorphisms of $\digamma(\mathbb{V})$ are
exactly the set of extensions of all permutations satisfying either
the condition (1) or (2) in Theorem \ref{mainth2} to the
automorphisms of $\digamma(\mathbb{V})$. Note that every such a
permutation extends (often not uniquely) to an automorphism of
$\digamma(\mathbb{V})$ and it is possible to count all such
extensions and hence all automorphisms of $\digamma(\mathbb{V})$ as
we will see in the following result.
 \end{remark}
 \begin{theorem}\label{permutSNS}Let $n\geq 3$ be a positive integer,
  then the following conditions hold:\\
  \indent \emph{\textbf{(1)}} Let $\mathcal{P}=\{\rho|\,\,
 \rho$ is an automorphism of the graph $\digamma(\mathbb{V}) $ such that $ \rho(\mathbb{V})=\mathbb{V}\}$.
 Then $|\mathcal{\mathcal{P}}|=
 (\frac{q^n-1}{q-1})!((q-1)!)^{2\frac{q^n-1}{q-1}}$.\\
\indent \emph{\textbf{(2)}} Let $\mathcal{T}=\{\rho|\,\,
 \rho$ is an automorphism of the graph $\digamma(\mathbb{V}) $
 such that $ \rho(\mathbb{V})=\mathbb{T}\}$.
 Then $|\mathcal{\mathcal{T}}|=
 (\frac{q^n-1}{q-1})!((q-1)!)^{2\frac{q^n-1}{q-1}}$.\\
 \indent \emph{\textbf{(3)}} $\digamma(\mathbb{V})$ has exactly
  $2(\frac{q^n-1}{q-1})!((q-1)!)^{2\frac{q^n-1}{q-1}}$ graph
  automorphisms.
 \end{theorem}
 \begin{proof} First remember that $|\Sigma|= |\{\mathcal{F}_H| H\in
 \Sigma\}|=\frac{q^n-1}{q-1}$.\par
 \textbf{ (1)} By Theorem \ref{mainth2}, $\rho\in \mathcal{P}$ if and only
 if $\rho$ acts as a permutation on $\Sigma$ and for any
  $H\in \Sigma$, $\rho(\mathcal{F}_H)$
 is determined uniquely by
 $\rho(\mathcal{F}_H)=\bigcap_{\mathcal{F}_H\subseteq
N(S)}N(\rho(S))$, say $\mathcal{F}_{H'}$ for a $H'\in\Sigma$. There
are $(\frac{q^n-1}{q-1})!$
 permutations on $\Sigma$. Note that for each pair $S, S'\in \Sigma
 $, there are exactly $(q-1)!$ bijections from $S$
 to $S'$. So for every permutation on $\Sigma$, there are exactly
 $((q-1)!)^{\frac{q^n-1}{q-1}}$ bijections from $\mathbb{V}$
  to $\mathbb{V}$. Hence
  all permutations on $\Sigma$ in total have
  $(\frac{q^n-1}{q-1})!((q-1)!)^{\frac{q^n-1}{q-1}}$ bijections
  (i.e. one to one functions) from $\mathbb{V}$ to $\mathbb{V}$.
On the other hand,  there are exactly $(q-1)!$ bijections from
$\mathcal{F}_H$ to $\mathcal{F}_{H'}$ (note that we fixed
$\rho(\mathcal{F}_H)=\mathcal{F}_{H'}$). Since
  $|\{\mathcal{F}_H| H\in \Sigma\}|=\frac{q^n-1}{q-1}$,
    all permutations $\rho$ on $\Sigma$ in total have exactly
 $((q$$-$$1)!)^{\frac{q^n-1}{q-1}}$ acceptable
bijections from  $\mathbb{T}$ to $\mathbb{T}$. By multiple principle
there are
 exactly $(\frac{q^n-1}{q-1})!((q$$-$$1)!)^{\frac{q^n-1}{q-1}}((q$$-$$1)!)^{\frac{q^n-1}{q-1}}$
   automorphisms of $\digamma(\mathbb{V})$ such a $\rho$,
    with $\rho(\mathbb{V})=\mathbb{V}$ and $\rho(\mathbb{T})=\mathbb{T}$.
    So $|\mathcal{P}|=(\frac{q^n-1}{q-1})!((q-1)!)^{2\frac{q^n-1}{q-1}}$.\par
\textbf{(2)} By Theorem \ref{mainth2}, $\rho\in \mathcal{T}$ if and
only if $\rho$ acts as a bijection
 $\Sigma$$\rightarrow$$\{\mathcal{F}_H|$$H\in$$\Sigma\}$ and for every $H\in\Sigma$
  $\rho(\mathcal{F}_H)$ is
  determined uniquely by
   $\bigcap_{\mathcal{F}_H\subseteq N(S)}N(\rho(S))$. For counting all elements of
 $\mathcal{T}$  we do similar as in the proof of the part (1). For each
pair $S,H\in \Sigma$ there exist
   $(q-1)!$ bijections from $S$ to $\mathcal{F}_H$.
   So for every fixed bijection from $\Sigma$ to $\{\mathcal{F}_H| H\in
   \Sigma\}$ there exist  $((q-1)!)^{\frac {q^n-1}{q-1}}$ bijections
   from $\mathbb{V}$ to $\mathbb{T}$.  There are exactly $(\frac{q^n-1}{q-1})!$
   bijections from $\Sigma$ to $\{\mathcal{F}_H| H\in\Sigma\}$,
    all of them in total have
 exactly $(\frac{q^n-1}{q-1})!((q-1)!)^{\frac{q^n-1}{q-1}}$
 bijections such that each of them sends
  $\mathbb{V}$ to $\mathbb{T}$. Note that since for every $H\in
  \Sigma$,
  $\rho(\mathcal{F}_H)$  is  determined uniquely, similar as the proof of part
  (1)there are exactly $((q-1)!)^{\frac {q^n-1}{q-1}}$ acceptable bijections
   such that each of them sends
   $\mathbb{T}$ to $\mathbb{V}$.
   Hence by multiple principle
  $|\mathcal{\mathcal{T}}|=(\frac{q^n-1}{q-1})!((q-1)!)^{2\frac{q^n-1}{q-1}}$.\par
 \textbf{(3)} By Remark \ref{a10} and Theorem \ref{mainth2} the cardinal number of the automorphism
  group of the graph $\digamma(\mathbb{V})$ is $|\mathcal{P}|+|\mathcal{T}|$
where the sets $\mathcal{P}$ and $\mathcal{T}$ are as in the parts
(1) and (2), respectively.
\end{proof}
In the following result the number of all automorphisms such a
$\rho$ of the graph $\digamma(\mathbb{V})$ is determined such that
$\rho(S)=S$ and $\rho(\mathcal{F}_H)=\mathcal{F}_H$ for all $S, H\in
\Sigma$. These kinds of automorphisms are exactly the set of all
automorphisms which stabilize the twin points( i.e. they are
permutations on twin points).
\begin{theorem}\label{identityrho} Let
$\mathcal{I}=\{\rho|\, \rho$ is an automorphism of the graph
$\digamma(\mathbb{V})$ and $\rho(S)$$=$$S$ for all
  $S\in \Sigma \}$.
 Then $|\mathcal{I}|= ((q-1)!)^{2\frac{q^n-1}{q-1}}$. \end{theorem}
\begin{proof} Assume $\rho\in \mathcal{I}$ is an automorphism of the graph
$\digamma(\mathbb{V})$. Then since $\rho(S)$$=$$S$ for all
$S\in\Sigma$, by Theorem \ref{mainth2}
$\rho(\mathcal{F}_H)$$=$$\mathcal{F}_H$ for all $H\in \Sigma$. Every
$S\in \Sigma$ has $(q-1)!$ permutations. Note that
$|\Sigma|=\frac{q^n-1}{q-1}$. So there are exactly
$((q-1)!)^{\frac{q^n-1}{q-1}}$  permutations such a $\rho$ on
$\mathbb{V}$ such that $\rho(S)=S$ for all $S\in\Sigma$. Similarly
there are $((q-1)!)^{\frac{q^n-1}{q-1}}$  permutation such a $\rho$
on $\mathbb{T}$ such that $\rho(\mathcal{F}_H)=\mathcal{F}_H$ for
all $H\in\Sigma$. So  by multiple principle  $\digamma (\mathbb{V})$
has exactly $((q-1)!)^{2\frac{q^n-1}{q-1}}$ automorphisms such a
$\rho$ such that $\rho(S)=S$ and $\rho
(\mathcal{F}_H)=\mathcal{F}_H$ for all $S, H\in \Sigma$. Hence
$|\mathcal{I}|=((q-1)!)^{2\frac{q^n-1}{q-1}}$
\end{proof}
\section{\bf Formolizing the automorphisms of $\digamma(\mathbb{V})$ }\
The aim of this section is decomposing the automorphisms of
$\digamma(\mathbb{V})$ to the combinations of known automorphisms.
In fact, the aim is formolizing the automorphisms.
\begin{lemma}Assume $n\geq3$, $B$$=$$\{e_1,e_2,\dots$$,e_n\}$ is the
standard basis of $\mathbb{V}_0$ and $\rho$ is an automorphism of
$\digamma(\mathbb{V})$. Then the set
$\rho(B)$$=$$\{\rho(e_1),\rho(e_2),\dots$ $,\rho(e_n)\}$  is
linearly independent.
\end{lemma}
\begin{proof} Note that by Lemma
 \ref{V or T}, $\rho(B)\subseteq
 \mathbb{V}$ or $\rho(B)\subseteq
 \mathbb{T}$. Assume $\rho(B)\subseteq
 \mathbb{V}$. With no lose of generality  define $\rho(0)=0$ and $\rho(f_0)=f_0$.
 By the contrary assume
 $\rho(B)=\{\rho(e_1),\rho(e_2),\dots,\rho(e_n)\}$ is linearly dependent.
 Then there exists a  nonzero $n\times n$ matrix $A$ on $\mathbb{F}_q$ such that
 $A\rho(e_i)=0$ for all $i=1,2,...,n$.  Define the matrix $\rho^{-1}(A)$
 whose $i$th row is  $\rho^{-1}(R_i)$ ,
  where $R_i$ is the $i$th row of $A$ as a transpose of a vector of $\mathbb{V}_0$,
   for $i=1,2,...,n$. Notice that $\rho^{-1}(R_i)$ is zero when $R_i$ is zero.
  Then since $\rho^{-1}$ is an automorphism
    of $\digamma(\mathbb{V})$,  $R_i\rho(e_1)=0$ implies
   $\rho^{-1}(R_i)(e_1)=0$, for all $i=1,2,...,n$ and hance the the
   first column of $\rho^{-1}(A)$ is zero. Then $R_i\rho(e_2)=0$ implies
   $\rho^{-1}(R_i)(e_2)=0$, for all $i=1,2,...,n$ and hance the the
   second column of $\rho^{-1}(A)$ is zero. Similarly we conclude
   that all columns of $\rho^{-1}(A)$ is zero, a contradiction, since $A$ has a nonzero row say
   $Rj$
   and $\rho^{-1}$ is an automorphism of $\digamma(\mathbb{V})$, so
   $\rho^{-1}(R_j)\neq0$. So $\rho(B)$ is  linearly independent.
 For the case $\rho(B)\subseteq \mathbb{T}$ similarly one can see
 that $\rho(B)$ is  linearly independent.
\end{proof}
\begin{theorem} \label{n=3theorem}For $n \geq 3$ let $\rho$ be a
bijective mapping on the vertex set of $\digamma(\mathbb{V})$. Then
$\rho$ is an automorphism of $\digamma(\mathbb{V})$ if and only if
exactly one of the following condition holds:\\
\indent\emph{\textbf{(1)}} $\rho(\mathbb{V})=\mathbb{V}$ and $\rho =
\chi_{P}\circ\pi \circ\tau$;\\
 \indent\emph{\textbf{(2)}} $\rho(\mathbb{V})=\mathbb{T}$ and $\rho =
\sigma \circ\chi_{P}\circ\pi \circ\tau$,\\
where $\chi_P$ is the regular automorphism induced by an invertible
matrix $P$, $\pi$ is the extending of a field automorphism of
$\mathbb{F}_q$ on $\digamma(\mathbb{V})$, $\tau$ is a permutation of
twin points of $\digamma(\mathbb{V})$ and $\sigma$ is the
automorphism with $\sigma(f_u)=u$ and $\sigma(u)=f_u$, for every
$u\in \mathbb{V}$.
 \end{theorem}
 \begin{proof} The sufficiency conditions is clear. So let $\rho$ be an
 automorphism $\digamma(\mathbb{V})$.
  By Lemma \ref{V or T} either $\rho(\mathbb{V})=\mathbb{V}$ or
 $\rho(\mathbb{V})=\mathbb{T}$.\\
\indent\textbf{Case 1:} Assume $\rho(\mathbb{V})=\mathbb{V}$. We
will show that the condition (1) holds.
 Let $P$$=$$(\rho(e_1), \rho(e_2), ... , \rho(e_n))$ be the block matrix
with $\rho(e_i)$ as its $i$th column and denote $\chi_{_{P^{-1}}}o
\rho$ by $\rho_1$. Then $\rho_1(e_i)$$=$$\chi_{_{P^{-1}}}o
\rho(e_i)$$=$$(P^{-1})\rho(e_i)$$=$$e_i$, for $i = 1, 2,..., n$. So
$\rho_1(e_i)$$=$$e_i$, for all $i$$=$$1, 2,..., n$. A nonzero vector
$u\in \mathbb{V}$ is  said to be monic if its first (from top)
nonzero entry is 1. Set $W=\{u\in\mathbb{V}|\, u$ is a monic
vector$\}$. Now, for every $v \in \mathbb{V}$, $\rho_1(v)=ru$ for a
scalar $r\in \mathbb{F}_q$ and a unique $u \in W$, which will be
written as $\rho_1(v)\equiv u$ (mod $\mathbb{F}_q$).  For every
$u\in W$ set $S_u=\{ru|\, 0\neq r\in \mathbb{F}_q\}$.  Then  $S_u\in
\Sigma$ is a 1-dimensional subspace without zero of $\mathbb{V}_0$.
Clearly if $u_1\neq u_2 \in W$, then $S_{u_1}\cap
S_{u_2}=\emptyset$. Also for every $t\in \mathbb{V}$,
$\rho_1(f_t)$$=$$\chi_{_{P^{-1}}}o\rho(f_t)$$=$$f_{ru}=rf_{u} $ for
a scalar $r\in \mathbb{F}_q$ and a unique $u\in W$, we say
$\rho_1(f_t)$$\equiv$$f_{u}$(mod $\mathbb{F}_q$). Also for every
$u\in W$, set $\mathcal{F}^{u}$$=$$\{f_{ru}|\, 0\neq r\in
\mathbb{F}_q\}$. Then $\mathcal{F}^{u}$ is a 1-dimensional subspace
without zero of $\mathbb{T}_0$. Also $u_1\neq u_2 \in W$, then
$\mathcal{F}^{u_1}\cap \mathcal{F}^{u_2}=\emptyset$. It is easy to
see that $\Sigma=\{S_u| u\in W\}$ and $\{\mathcal{F}^u| u\in
W\}=\{\mathcal{F}_H| H\in \Sigma\}$. Similar as the proof of
{\cite[Theorem 4.1]{XWang1}} we will continue the proof with some
claims. Some of climes are  similar as those of {\cite[Theorem
4.1]{XWang1}} and we will point to the similarities
 every where we use. Note that for every $v=\sum_{i=1}^n
a_ie_i \in \mathbb{V}$, $f_v=f_{(\sum_{i=1}^n a_ie_i)}= \sum_{i=1}^n
a_if_{e_i}$. In fact considering {\cite[Theorem 4.1]{XWang1}}, here
to have a similar method to prove our theorem for the graph
$\digamma(\mathbb{V})$, we have replaced $E_{1i}$ the elementary
matrix on $\mathbb{F}_q$, by the linear functional $f_{e_i}\in
\mathbb{T}$, for $i=1,2,\ldots,n$. Consequently, we have to replace
the matrix $\sum_{i=1}^n a_i E_{1i}$ by the linear functional
$f_{(\sum_{j=1}^n a_ie_i)}$$=$$\sum_{i=1}^n a_if_{e_i}$. So
consequently, in the proof of {\cite[Theorem 4.1]{XWang1}}, one can
see that in this paper we have to replace the notation (mod $G$)
with notation (mod $\mathbb{F}_q$), where $G$ is the group of
$n\times n$ invertible
matrices on $\mathbb{F}_q$. So by the mentioned replacements, we have a few claims below:\\
 \textbf{Claim 1}. For $v$$=$$\sum_{i=1}^n
a_ie_i$$\in$$W$, assume $\rho_1 (\sum_{i=1}^n
a_if_{e_i})$$\equiv$$\sum_{i=1}^n a'_if_{e_i}$ (mod $\mathbb{F}_q$),
with $\sum_{i=1}^n a_i'e_i\in W$ the set introduced before the
claim. Then
 $a_i=0$ if and
 only if $a_i'=0$. In particular, $\rho_1(f_{e_i})$$\equiv$$f_{e_i}$ (mod
 $\mathbb{F}_q$). Similarly for $\sum_{i=1}^n a_ie_i$$\in$$W$, assume that $\rho_1
(\sum_{i=1}^n a_ie_i)$$\equiv$$\sum_{i=1}^n a'_ie_i$ (mod
$\mathbb{F}_q$), with $\sum_{i=1}^n a_i'e_i$$\in$$W$, then $a_i=0$
if and only if $a_i'=0$.\par
 The proof is similar as  that of
{\cite[Claim 3]{XWang1}}.\par
 So  for any $a\in \mathbb{F}_q$,
 there exists a unique $a'\in \mathbb{F}_q$ such that
 $\rho_1(e_i + ae_j)\equiv e_i+ a'e_j$ (mod $\mathbb{F}_q$).
  Thus for $1\leq i<j\leq n$, we can define a permutation $\pi_{ij}$ on $\mathbb{F}_q$
   such that $\pi_{ij}(0) = 0$ and $\rho_1(e_i +
ae_j)\equiv e_i+\pi_{ij}(a)e_j$ (mod $\mathbb{F}_q$). Similarly  for
$1\leq i<j\leq n$, we can define a permutation $\theta _{ij}$ on
$\mathbb{F}_q$
   such that $\theta_{ij}(0) = 0$ and $\rho_1(f_{e_i +
ae_j})$$\equiv$$f_{e_i+\theta_{ij}(a)e_j}$ (mod $\mathbb{F}_q$).
From $f_{e_i-a^{-1} e_j}\sim e_i+ae_j$ it follows that
 $f_{e_i +\theta_{ij}(-a^{-1})e_j}\sim e_i+\pi_{ij}(a)e_j$.
  Then similar as {\cite[Claim
5]{XWang1}} for every $0\neq a\in F_q$, we have
$\theta_{ij}(-a^{-1})\pi_{ij}(a)=-1$,
$\pi_{ij}(-a^{-1})\theta_{ij}(a)=-1$.\\
\textbf{Claim 2}. For $1\leq$$i$$<$$j$$\leq$$n$, we have
$\rho_1(f_{e_i+\sum_{j=i+1}^n a_je_j})$$\equiv$$f_{e_i +\sum^ n
_{j=i+1} \theta_{ij}(a_j)e_j}$ (mod $\mathbb{F}_q$) and $\rho_1(e_i
+\sum_{j=i+1}^n a_je_j)\equiv e_i +\sum_{j=i+1}^n \pi_{ij}(a_j)e_j$
(mod $\mathbb{F}_q$).\par
 The proof is similar as that of {\cite[Claim 6]{XWang1}}.\par
\textbf{Claim 3}. For $2\leq i<j\leq$$n$ and $a\in\mathbb{F}_q$, we
have $\theta_{ij}(a)$$=$$\pi_{1i}(1)\theta_{1j}(a)$ and
$\pi_{ij}(a)$$=$$\theta_{1i}(1)\pi_{1j}(a)$.\par
The proof is similar as that of {\cite[Claim 7]{XWang1}}.\\
\textbf{Claim 4}. For $1$$\leq$$i<j$$\leq$$n$, we have
$\rho_1(e_i+\sum^n_{j=i+1}
a_je_j)$$\equiv$$e_i+\theta_{1i}(1)\sum^n_{ j=i+1}\pi_{1j}(a_j)e_j$
(mod $\mathbb{F}_q$) and $\rho_1(f_{e_i+\sum^n_{j=i+1}
a_je_j})\equiv f_{e_i +
\pi_{1i}(1)\sum^n_{j=i+1}\theta_{1j}(a_j)e_j}$ (mod $\mathbb{F}_q$),
where  $\theta_{11}(1)=\pi_{11}(1)=1$.\par
 The proof is similar as that of {\cite[Claim 8]{XWang1}}.\par
 Now, let  $\pi$ the function on $\mathbb{F}_q$
 defined as $\pi(a)=\frac{ \pi_{12}(a)}{\pi_{12}(1)}$ for any $a\in
\mathbb{F}_q$. Then similar as {\cite[Claim 9 and Clime
10]{XWang1}},  $\pi$ is a field automorphism of $\mathbb{F}_q$.
Let $Q=\emph{diag}(1, \pi_{12}(1), ... , \pi_{1n}(1))$ be a diagonal
matrix. Then by explaining after {\cite[Claim 10]{XWang1}},
$\pi^{-1}\circ \chi_{_{Q^{-1}}}\circ\rho_1$ sends any $u\in W$ to a
nonzero scalar multiple of $u$.  Denote $\pi^{-1} \circ
\chi_{_{Q^{-1}}}\circ\rho_1$ by $\rho_2$ for the convenience. Hence
 for every $u\in W$, $\{\rho_2(u), u\}\subseteq S_u$. So
 by Lemma \ref{a8},  $\rho_2(S_u)=S_u$, as every element $v\in S_u$
 and $u$ are linearly dependent and so $\rho_2(v)\in S_u$.
  So $\rho_2(S)=S$ for every $S\in \Sigma$. Then by Lemma
   \ref{intersection} $\rho_2(\mathcal{F}_H)= \bigcap_{\mathcal{F}_H\subseteq
N(S)}N(\rho_2(S))=\bigcap_{\mathcal{F}_H\subseteq
N(S)}N(S)=\mathcal{F}_H$ for all $H\in \Sigma$. Therfore $\rho_2$ is
a permutation of twin points of $\digamma(\mathbb{V})$.

Finally, in this case we have
$\pi^{-1}\circ\chi_{Q^{-1}}\circ\chi_{P^{-1}}\circ\rho = \tau$,
where $\tau$ is a permutation of twin points. Thus $\rho =
\chi_{P_0}\circ\pi\circ\tau$, where $P_0=PQ$.\par
 \indent\textbf{Case 2:} Assume  $\rho(\mathbb{V})=\mathbb{T}$.
 Then consider the permutation $\sigma$ on the vertex set of $\digamma(\mathbb{V})$
 such that $\sigma (f_u)=u$ and $\sigma(u)=f_u$ for every $u\in
 \mathbb{V}$.
 Then for every $u,v\in \mathbb{V}$,
 $v\sim f_u$ if and only if $f_v\sim u$ if and only if $\sigma (v)\sim\sigma
 (f_u)$. So $\sigma$ is an automorphism of $\digamma(\mathbb{V})$.
 Then, $\sigma^{-1}\circ\rho(\mathbb{V})=\mathbb{V}$ and
 $\sigma^{-1}\circ\rho$ is an automorphism of the graph $\digamma(\mathbb{V})$.
 So replacing $\rho$ by $\sigma^{-1}\circ\rho$ in part (1),
 implies $\sigma^{-1}\circ\rho=\chi_{P_0}\circ\pi\circ\tau$  and
  hence $\rho=\sigma\circ\chi_{P_0}\circ\pi \circ\tau$.
Hence the condition (2) holds. The proof is complete.
\end{proof}
Now, for the graph $\digamma(\mathbb{V})$ where $n=2$  the following
remark is needed.
\begin{remark} \label{sigma n=2} Assume $n=2$ and $\rho$ is
 an automorphism of $\digamma(\mathbb{V})$. Then by Remark \ref{n=2}
 and Lemma \ref{Lemn=2}, it is possible that for two $S_1\neq S_2\in
 \Sigma$, we have $\rho(S_1)\subseteq \mathbb{T}$ and $\rho(S_2)\subseteq
 \mathbb{V}$(i.e. $\rho(\mathbb{V})\notin \{\mathbb{V}, \mathbb{T}\}$). Now, define a permutation $\delta$ on the vertex set of
 $\digamma(\mathbb{V})$ such that:\\
  $\delta(f_u)$$=$$\left\{%
\begin{array}{ll}
    u, &\hbox{if $\rho(v)$$=$$f_u$, for a $v$$\in$$\mathbb{V}$;} \\
    f_u, &\hbox{if $\rho(f_v)$$=$$f_u$, for a $v$$\in$$\mathbb{V}$} \\
\end{array}%
\right.$ and\,
 $\delta(v)$$=$$\left\{%
\begin{array}{ll}
    f_v, &\hbox{if $\rho(f_u)$$=$$v$, for a $u$$\in$$\mathbb{V}$;} \\
    v, &\hbox{if $\rho(u)$$=$$v$, for a $u$$\in$$\mathbb{V}$.} \\
\end{array}%
\right.$\\ It is easy to see that
$\delta^{-1}\circ\rho(\mathbb{V})=\mathbb{V}$. Then $\delta$ is an
automorphism of $\digamma(\mathbb{V})$.
\end{remark}
 \begin{proof} Assume $f_x\sim y$. Then we have two cases:\\
\indent \textbf{Case 1)} Since $\rho$ is an an automorphism then,
$f_x=\rho(v)$, for a $v\in \mathbb{V}$, if and only if
$y=\rho(f_t)$, for a $t\in \mathbb{V}$.  Hence $\delta(f_x)=x$ if
and only if $\delta(y)=f_y$. Note
that $x\sim f_y$. Therefore $f_x\sim y$ if and only if $\delta(f_x)\sim\delta(y)$.\\
\indent\textbf{Case 2)} Similar as the Case 1, $f_x=\rho (f_v)$ for
a $v\in \mathbb{V}$ if and only if $\rho (y)=t $, for a $t\in
\mathbb{V}$.  Hence $\delta(f_x)=f_x$ if and only if $\delta(y)=y$.
 Therefore $f_x\sim y$ if and only if
 $\delta(f_x)\sim\delta(y)$.\par
 So $\delta$ is an automorphism of the graph $\digamma(\mathbb{V})$.
\end{proof}

\begin{remark}\label{phibar} Assume $n=2$ and $\phi$ is a
  permutation on $\mathbb{F}_q$ such that $\phi(0)=0$.
   We define a mapping $\overline{\phi}$ on the vertex set
of $\digamma(\mathbb{V})$ as follows for any  vectors $u=ae_1+be_2$
and $v=ce_1+de_2$ of $\mathbb{V}$:\\
                 $\overline{\phi}(f_u)$$=$$\left\{%
                         \begin{array}{ll}
           af_{e_1}$$+$$a\phi(a^{-1}b)f_{e_2},&\hbox{if $\emph{a}$$\neq$$0$} \\
                       f_u,& \hbox{if \emph{a}$$=$$0} \\
                            \end{array}%
                            \right.$ ,
                                 $\overline{\phi}(v)$$=$$\left\{%
                         \begin{array}{ll}
           ce_1$$-$$c\phi($$-$$cd^{-1}$$)^{-1}e_2,&\hbox{\noindent if $\emph{cd}\neq0$} \\
                       v, & \hbox{if \emph{cd}$$=$$0} \\
                            \end{array}%
                            \right.$\\
Note that if $f_u(v)=ac+bd=0$, where $ac$ and $bd$ both are nonzero,
then $a^{-1}b$$+$$cd^{-1}$$=$$0$ and hence $a^{-1}b$$=$$-cd^{-1}$.
So considering this relation,  one can see that $\overline{\phi}$ is
an automorphism of the graph $\digamma(\mathbb{V})$.
\end{remark}
The following result states the form of automorphisms for
$\digamma(\mathbb{V})$, where  $n=2$.
\begin{theorem}\label{formn=2} For $n$$=$$2$, let $\rho:V\rightarrow V$ be a
bijection
 on the vertex set of $\digamma(\mathbb{V})$. Then $\rho$ is an
automorphism of $\digamma(\mathbb{V})$ if and only if
$\rho=\delta\circ\chi_P \circ\overline{\phi}\circ\tau$, where
$\chi_P$ is the regular automorphism induced by an invertible matrix
$P$, $\phi$ is a permutation on $\mathbb{F}_q$ with $\phi(0)=0$,
$\overline{\phi}$ and $\delta$ are as just defined  and $\tau$ is a
permutation of twin points of $\digamma(\mathbb{V})$.
\end{theorem}
\begin{proof}
Note that  $\rho(\mathbb{V})=\mathbb{V}$ if and only if  $\delta$ is
identity. Then we have two cases as following:\\
\indent \textbf{Case 1)}  First we  assume $\delta$ is identity and
hence $\rho(\mathbb{V})=\mathbb{V}$. Similar as in the proof of
Theorem \ref{n=3theorem}, there exists a $2\times 2$ invertible
matrix $P$ such that $\chi_{_{P^{-1}}}\circ\rho(e_i)$$=$$e_i$ for
every $i=1, 2$. Denote $\chi_{_{P^{-1}}}\circ\rho$ by $\rho_1$. Then
similar as the proof of {\cite[Theorem 5.1]{XWang1}}, for any $a\in
\mathbb{F}_q$ there exists  a unique $a'\in \mathbb{F}_q$ such that
$\rho_1(f_{e_1}$$+$$af_{e_2})\equiv f_{e_1}$$+$$a'f_{e_2}$ (mod
$\mathbb{F}_q$). Thus $\rho_1$ induces a permutation $\phi$ on
$\mathbb{F}_q$ such that $\phi(0)=0$ and
$\rho_1(f_{e_1}+af_{e_2})$$\equiv$$f_{e_1}$$+$$\phi(a)f_{e_2}$ (mod
$\mathbb{F}_q$), and hence
$\rho_1(f_{e_1+ae_2})$$\equiv$$f_{e_1+\phi(a)e_2}$(mod
$\mathbb{F}_q$).  So the set $\{\rho_1(f_{e_1+ae_2}),
f_{e_1+\phi(a)e_2}\}$ is linearly dependent. Similarly, $\rho_1$
induces a permutation $\theta$ on $\mathbb{F}_q$ such that
$\theta(0)=0$ and $\rho_1(e_1+ae_2)$$\equiv$$e_1$$+$$\theta(a)e_2$
(mod $\mathbb{F}_q$). So the set $\{\rho_1(e_1+ae_2),
e_1$$+$$\theta(a)e_2\}$ is also linearly dependent.
 Since $\rho_1$ is an automorphism, for every element $0\neq a\in\mathbb{F}_q$,
 the adjacency $f_{e_1-a^{-1}e_2}$$\sim$$e_1+ae_2$
implies $\rho(f_{e_1-a^{-1}e_2})$$\sim$$\rho(e_1+ae_2)$. Hence
$f_{e_1+\phi( -a^{-1} )e_2}\sim e_1+\theta(a)e_2$ (as both of the
sets $\{\rho(e_1+ae_2), e_1+\theta(a)e_2\}$ and
$\{\rho(f_{e_1-a^{-1}e_2}), f_{e_1+\phi( -a^{-1} )e_2} \}$ are
linearly dependent and hence by Lemma \ref{a6} and Lemma \ref{a5}
they consist of twin points). Hence $\theta(a)=-\phi(-a^{-1})^{-1}$
for all $0\neq a\in \mathbb{F}_q$. So $\rho_1(e_1+ae_2)\equiv
e_1-\phi(-a^{-1})^{-1}e_2$ (mod $\mathbb{F}_q$).  Then the function
$\overline{\phi}$ as in the remark \ref{phibar} is an automorphism
of $\digamma(\mathbb{V})$.
 One can see that for all $a\in\mathbb{F}_q$:\par
 \noindent$(\overline{\phi^{-1}}\circ\rho_1)(e_1$$+$$ae_2)$$\equiv$$e_1+ae_2$ (mod
$\mathbb{F}_q$)\,\,\ and \,\
    $(\overline{\phi^{-1}}\circ\rho_1)(f_{e_1+ae_2})$$\equiv$$f_{e_1+ae_2}$ (mod
$\mathbb{F}_q$). Let $\tau= \overline{\phi^{-1}}\circ\rho_1$.
Similar as Claim 5 in the proof of Theorem \ref{n=3theorem} one can
see that $\tau$ is a permutation of twin points of the graph
$\digamma(\mathbb{V})$. Therefore $\rho=\chi_{_P}
\circ\overline{\phi}\circ\tau$.\par
 \indent \textbf{Case 2)} If $\delta$ is not identity then by Remark
 \ref{sigma n=2} $\delta^{-1}\circ\rho
(\mathbb{V})=\mathbb{V}$. Then in Case 1, one can replace $\rho$ by
the automorphism
  $\delta^{-1}\circ\rho$ of $\digamma(\mathbb{V})$. Then we have $\rho=\delta\circ \chi_P
  \circ \overline{\phi}\circ \tau$ and the proof is complete.
\end{proof}
\bibliographystyle{elsarticle-num}

\begin{thebibliography}{99}
 \bibitem{ak2} S. Akbari, M. Habibi, A. Majidinya, R. Manaviyat, The inclusion ideal graph of rings,
  Comm. Algebra, 43(6) (2015) 2457-2465, DOI: 10.1080/00927872.2014.894051.
\bibitem{Das} A. Das,  Non-zero component union
graph of a fnite-dimensional vector space, Linear Multilinear
Algebra, 65(6)(2017) 1276-1287, DOI:10.1080/03081087.2016.1234577.
\bibitem{Gu3} Z. Gu, Z. Wan, Automorphisms of
subconstituents of symplectic graphs. Algebra Colloquium. 20(2)
(2013)pp.333–342, DOI: 10.1142/S1005386713000308.
\bibitem{Gu1} Z. Gu, Z. Wan, Orthogonal graphs of
odd characteristic and their automorphisms. Finite Fields Appl.
14(2) (2008) 291-313, DOI:10.1016/j.ffa.2006.12.001.
\bibitem{Gu2} Z. Gu, Z. Wan, Subconstituents of orthogonal
graphs of odd characteristic. Linear Algebra Appl. 439(10) (2011)
2430-2447, DOI:10.1016/j.laa.2013.08.010.
\bibitem{Hoffman} K. Hoffman, R. Kunze, Linear Algebra,
Prentice-Hall,  New Jersey, 1971.
\bibitem{Wliu1} W. Liu, C. Ma, K. Wang. Full automorphism group of generalized
unitary graphs. Linear Algebra Appl. 437(2) (2012) 684-691,
DOI:10.1016/j.laa.2012.02.024.
\bibitem{Matczuk1} J. Matczuk, M. Nowakowska, E. R Puczy{\l}owski,
  Intersection graphs of modules and rings, J. Algebra Appl. 17(7) (2018) 1850131,
  DOI: 10.1142/S0219498818501311.
\bibitem{XWang1} X. Wang, D. Wong, D. Sun,  Automorphisms and
domination numbers of transformation graphs over vector spaces,
Linear and Multilinear Algebra, (2018)1350-1363, DOI:
10.1080/03081087.2018.1452890
\bibitem{ZWan1} Z. Wan, K. Zhou, Unitary graphs and their
automorphisms. Ann Comb. 14(3)(2010) 367-395.
\bibitem{LWang1} L. Wang, A note on
automorphisms of the zero-divisor graph of upper triangular
matrices. Linear Algebra Appl. 465 (2015) 214-220,
DOI:10.1016/j.laa.2014.09.035.
 \bibitem{LWang2}L. Wang,
Automorphisms of the zero-divisor graph of the ring of matrices over
a fnite field. Discrete Math. 339(8) (2016) 2036-2041, DOI:
10.1016/j.disc.2016.02.021.
\bibitem{wang1} L. Wang, Domination numbers and automorphisms of dual
 graphs over vector spaces, Bull. Malays. Math. Sci. Soc., 43 (2020), 689-701,
  Doi:10.1007/s40840-018-00709-1.
\bibitem{wang2} H. J. Wang, Co-maximal graph of non-commutative rings,
 Linear Algebra Appl. 430 (2009) 633-641, DOI: 10.1016/j.laa.2008.08.026.
\bibitem{wong1} D. Wong, X. Ma,J. Zhou, The group of
automorphisms of a zero-divisor graph based on rank one upper
triangular matrices. Linear Algebra Appl. 460(2014)242-258, DOI:
10.1016/j.laa.2014.07.041.
 \end{thebibliography}
 

 \end{document}